\documentclass[12pt]{amsart}
\usepackage{amsmath}
\usepackage{amsfonts}
\usepackage{amsthm}
\usepackage{amssymb}
\usepackage{amscd}
\usepackage[all]{xy}
\usepackage{enumerate}



\theoremstyle{plain}
\newtheorem{thm}{Theorem}[subsection]

\newtheorem{prop}[thm]{Proposition}

\newtheorem{cor}[thm]{Corollary}

\newtheorem{lem}[thm]{Lemma}

\theoremstyle{definition}
\newtheorem{defn}[thm]{Definition}

\newtheorem*{ackn}{Acknowledgment}

\newtheorem{rmk}[thm]{Remark}
\newcommand{\sA}{\mathcal{A}}
\newcommand{\sB}{\mathcal{B}}
\newcommand{\sC}{\mathcal{C}}
\newcommand{\sD}{\mathcal{D}}
\newcommand{\sE}{\mathcal{E}}
\newcommand{\sF}{\mathcal{F}}

\newcommand{\sI}{\mathcal{I}}

\newcommand{\sK}{\mathcal{K}}
\newcommand{\sL}{\mathcal{L}}

\newcommand{\sO}{\mathcal{O}}

\newcommand{\sX}{\mathcal{X}}

\newcommand{\sZ}{\mathcal{Z}}

\newcommand{\mC}{\mathbb{C}}

\newcommand{\mP}{\mathbb{P}}

\newcommand{\mZ}{\mathbb{Z}}
\newcommand{\mR}{\mathbb{R}}

\newcommand{\Ker}{\mathrm{Ker}\,}

\usepackage{color}
\newcommand{\ThfC}[2]{\mathrm{\theta_{#1}^{#2}}} 
\newcommand{\Ab}{\mathrm{A}} 
\newcommand{\AD}{\mathcal{D}} 
\newcommand{\AL}{\mathcal{L}} 
\newcommand{\LATT}{\Lambda} 
\newcommand{\VSA}{V} 
\newcommand{\ConnHom}{f}    
\newcommand{\bigslant}[2]{{\raisebox{.2em}{$#1$}\left/\raisebox{-.2em}{$#2$}\right.}}  

\newcommand{\IMh}[1]{\Im m \ #1} 

\numberwithin{equation}{section}

\newcommand{\beba}  {\begin{equation}\begin{array}{rcl}}

\newcommand{\eaee}  {\end{array}\end{equation}}

\makeatletter
\def\l@section{\@tocline{1}{0pt}{1pc}{}{}}
\def\l@subsection{\@tocline{2}{0pt}{1pc}{4.6em}{}}
\def\l@subsubsection{\@tocline{3}{0pt}{1pc}{7.6em}{}}
\renewcommand{\tocsection}[3]{%
  \indentlabel{\@ifnotempty{#2}{\makebox[2.3em][l]{%
    \ignorespaces#1 #2.\hfill}}}#3}
\renewcommand{\tocsubsection}[3]{%
  \indentlabel{\@ifnotempty{#2}{\hspace*{2.3em}\makebox[2.3em][l]{%
    \ignorespaces#1 #2.\hfill}}}#3}
\renewcommand{\tocsubsubsection}[3]{%
  \indentlabel{\@ifnotempty{#2}{\hspace*{4.6em}\makebox[3em][l]{%
    \ignorespaces#1 #2.\hfill}}}#3}
\makeatother

\title{On Birationally trivial families and Adjoint quadrics}

\author{Luca Cesarano}
\address{Mathematisches Institut\\
Universit\"at Bayreuth\\
\texttt{luca.cesarano@uni-bayreuth.de}}

\author{Luca Rizzi}
\address{DMIF \\
Universit\`a di Udine\\
Udine, 33100 Italia\\
\texttt{rizzi.luca@spes.uniud.it}}

\author{Francesco Zucconi}

\address{DMIF \\
Universit\`a di Udine\\
Udine, 33100 Italia\\
\texttt{Francesco.Zucconi@dimi.uniud.it}}

\begin{document}

\markboth{Rizzi and Zucconi}{Birationally trivial families and Adjoint quadrics }

\begin{abstract} Let $\pi\colon \sX\to B$ be a family whose general fiber $X_b$ gives a $(d_1,...,d_a)$ polarisation of a general Abelian variety where $1\leq d_i\leq 2$, $i=1,...,a$ and $a\geq 4$. We show that the fibers are in the same birational class if all the $(m,0)$ forms on $X_b$ are liftable to $(m,0)$ forms on $\sX$ where $m=1$ and $m=a-1$. Actually we show general criteria to find families with fibers in the same birational class,  which leads together with a famous theorem of Nori to some interesting applications.
 \end{abstract}
\maketitle

\section{Introduction}
A family of $n$-dimensional varieties is a flat, smooth 
proper morphism $\pi\colon\sX\rightarrow B$ such that the {\it{fiber}} $X_b :=\pi^{-1}(b)$ over a point  $b\in B$ has dimension $n$. In this paper, we only assume that $B${\it{ is a smooth connected open variety of dimension $1$.}} We will also assume that $X_{b}$ is an 
irregular smooth variety of general type such that its Albanese morphism ${\rm{alb}}(X_b):X_{b}\rightarrow {\rm{Alb}}(X_{b})$ is of degree $1$. We want to study conditions which ensure that the fibers of $\pi\colon\sX\rightarrow B$ have the same birational type. 

It is well-known that, up to base change, we can associate to  $\pi\colon\sX\rightarrow B$ the family of corresponding Albanese varieties. In fact, we can work in the more general set up of families of Albanese type presented in \cite[Definition 1.1.1]{PZ}. We recall below the basic definition.

Let 
$p\colon \sA\rightarrow B$ be a family of Abelian varieties; that is, the fiber $A_{b}:=p^{-1}(b)$ is an Abelian variety
of dimension $a>0$. We say that a morphism
$\Phi\colon \sX{\rightarrow}\sA$ is a {\it family of Albanese type} over $B$ if:
\begin{enumerate} 
\item $\Phi$ fits into the
commutative diagram:
$$
\xymatrix{
& \sX \ar[dl]_{\Phi} \ar[dr]^{\pi}&& 
 & \\
 \sA&\stackrel{p}{\longrightarrow} & B & &  }
$$
\item the induced map
$\phi_{b}\colon X_{b}\rightarrow A_{b}$ of $\Phi$ on
$X_{b}$ is birational onto its image $Z_{b}$;
\item the cycle $Z_{b}$ generates the fiber $A_{b}$ as a group.
\end{enumerate}

A family of Albanese type comes equipped with a global object: its relative homologically trivial cycle. Indeed let
$-{\rm{Id}}_{\sA}\colon \sA\rightarrow\sA$ be the natural involution induced on
$p:\sA\rightarrow B$ by the multiplication by $(-1)$ on the fibers. The
composition $(-{\rm{Id}}_{\sA})\circ\Phi\colon \sX\rightarrow\sA$ is an Albanese type family. We set
$(-{\rm{Id}}_{\sA})\circ\Phi:=\Phi^{-}$. Then we can construct two cycles $[\Phi\colon \sX\rightarrow \sA]$ and $[\Phi^{-}\colon\sX\rightarrow \sA]$ in the relative group $Z^{a-n}_{}(\sA/B)$.
We define: $[\sX]^{+}:=[\Phi\colon \sX\rightarrow \sA]$ and
$[\sX]^{-}:=[\Phi^{-}:\sX\rightarrow
\sA]$. The following will be called
the {\it{basic cycle of the Albanese type family}} $\Phi\colon \sX{\rightarrow}\sA$:
\begin{equation}\label{cnstccl}
[\sZ]=[\sX]^{+}-[\sX]^{-}.
\end{equation}\noindent
It is well-known that the cycle $[\sZ]$ is relatively homologically trivial; that is $[\sZ]\in Z^{a-n}_{h}(\sA/B)$. By the theory of normal functions and its infinitesimal invariant $\delta_{\sZ}$, see: \cite{grif1}, \cite{grif2}, \cite{Vo}, we can split out Albanese type families into two types: those whose  infinitesimal invariant is non zero and those which have $\delta_{\sZ}=0$. These last are called {\it{Nori trivial families}}.

Another information carried by the morphism $\phi_{b}\colon X_{b}\rightarrow A_{b}$ is a splitting of $H^{n,0}(X_b)$. Indeed let $\phi_b^{\star} \colon H^{n,0}(A)\to H^{n,0}(X_b)$ and set $V_b:= {\rm{Im}}(\phi_b^{\star})$. Inside the dual $H^{0,n}(X_b)$ of $H^{n,0}(X_b)$ we can define:
$$
{\rm{Ann}}(V_b):=\{\tau\in H^{0,n}(X_b)\mid \int_{X_b} \phi_{b}^{\star}(\mu)\wedge\tau=0, \,\,\forall \mu\in H^{n,0}(A) \}
$$
and we know that 
\begin{equation}\label{barra}
H^{0,n}(X_b)={\overline{V_b}}\oplus {\rm{Ann}}(V_b)
\end{equation}
where ${\overline{V_b}}\subset H^{0,n}(X)$ is the conjugate space of $V_b$. It also holds:
\begin{equation}\label{senzabarra}
H^{n,0}(X_b)=V_b\oplus {\overline{{\rm{Ann}}(V_b)}}
\end{equation}
The standard multiplication map $H^{n,0}(X_b)\otimes H^{n,0}(X_b)\to H^0(X,\omega_{X_b}^{\otimes 2})$ induces an homomorphism
\begin{equation}\label{iniettivobar}
\nu_{X_b}\colon {\rm{Sym}}^{2}{\overline{{\rm{Ann}}(V_b)}}\to H^0(X,\omega_{X_b}^{\otimes 2})
\end{equation}

Finally, we say that a family of relative dimension $n$ satisfies {\it{extremal liftability assumptions}} if the natural restriction homomorphisms $H^0({\sX},\Omega_{{\sX}}^1)
\to H^0(X_{b},\Omega_{X_{b}}^1)$ and $H^0({\sX},\Omega_{{\sX}}^n)
\to H^0(X_{b},\Omega_{X_{b}}^n)$ are surjective for every $b\in B$. We show:

\medskip

\noindent
{\bf{Main Theorem}} {\it{Let  $\Phi\colon\sX\rightarrow\sA$ be a Nori family. If it satisfies extremal liftability assumptions and $\nu_{X_b}\colon {\rm{Sym}}^{2}{\overline{{\rm{Ann}}(V_b)}}\to H^0(X,\omega_{X_b}^{\otimes 2})$ is injective for a general $b\in B$, then its fibers belong to the same birational class.}}
\medskip

The proof is a direct consequence of the new notion of adjoint quadric introduced in \cite{RZ}. By extremal liftability assumptions we are actually concerned on families of varieties equipped with a morphism to a fixed abelian variety; see  Proposition \ref{insalata}. Nevertheless our result should be seen in the light of the theory of families of varieties of general type as described in \cite{Ko}; this is the reason why we present the theorem in the above general set up. We strongly rely on the theory exposed in \cite{PZ}. In particular, we use the Volumetric Theorem \cite[Theorem 1.5.3]{PZ}. 

As an immediate consequence of the Main Theorem and of the well-known fact that if $C$ is a hyperelliptic curve then $C-C^{-}$ is trivial in its Jacobian, we have the well-known Torelli Theorem for hyperelliptic deformations of hyperelliptic curves, see: \cite{OS}; this is a case where ${\overline{{\rm{Ann}}(V_b)}}=0$.

More deeply, by a famous Theorem of Nori \cite [pp 372]{N}, see also \cite{Fa}, our Main Theorem applies to the case where the family $\sX$ is given by a family of cycles inside a general Abelian variety of dimension $a\geq 4$ such that the homomorphism $\nu_{X_b}\colon {\rm{Sym}}^{2}{\overline{{\rm{Ann}}(V_b)}}\to H^0(X,\omega_{X_b}^{\otimes 2})$ is injective. In particular we have:

\newpage
\noindent {\bf{Theorem [A]}} {\it{Let $(A,\sL)$ be a $(d_1,d_2,..., d_a)$ polarised Abelian variety which is general inside its moduli space and $1\leq d_i\leq 2$, $i=1,..., a$. Let 
$\pi\colon \sX\to B$ be a family which satisfies extremal liftability conditions and such that its general fiber admits a degree-$1$ morphism $\phi_b\colon X_b\to A$. If the image $Z_b$ is an element of the linear system $|\sL|$ then the fibers belong to the same birational class.}}

Theorem [A], see: Subsection (\ref{Theorem A}), follows by the Main Theorem and by a careful study of the map $\nu_{X_b}\colon {\rm{Sym}}^{2}{\overline{{\rm{Ann}}(V_b)}}\to H^0(X,\omega_{X_b}^{\otimes 2})$ where $X_b$ is an element of the linear system associated with a $(d_1,d_2,..., d_a)$ polarisation. We expect further generalisations of Theorem [A].

Finally, we have an application to the case of fibrations with maximal relative irregularity of the above circle of ideas.

Let $S$, $B$ be respectively a smooth surface and a smooth curve. A fibration $f\colon S\to B$ is said to be of maximal relative irregularity if $q(S)-g(B)=g(F)-1$ where $q(S)$ is the irregularity of $S$ and $g(B),g(F)$ are respectively the genus of $B$ and of a general fiber $F$.
There are many papers on this topic. Here we can quote \cite{P}, \cite{M} and \cite{BGN} which also contains basic references to this problem. In this case we have an occurrence of  the case where ${\overline{{\rm{Ann}}(V_b)}}\simeq \mathbb C$. Indeed if the natural morphism $F\to {\rm{Alb}}(S)$ has degree $1$ then we can find a suitable open subscheme $U\subset B$ contained in the locus where $f\colon S\to B$ is smooth, and we can form a family of Albanese type $\Phi_U\colon S_U\to\sA_U$ where $p\colon \sA\to U$ is such that all the fibers are isomorphic to a fixed Abelian variety $A$ of dimension $g(F)-1$ and $S_U:=f^{-1}(U)$. In Theorem \ref{vaiinfine}, which does not depend on the Main Theorem, we show that the infinitesimal invariant associated to the basic cycle associated to $f_{|S_U}\colon S_U\to U$ is not zero.

\begin{ackn}
This research is supported by DIMAGeometry PRIDZUCC. The authors want to thank Gian Pietro Pirola because he pointed out two unclear deductions  contained in a preliminary version of this work  and Claire Voisin for her useful advices. They also want to thank Ingrid Bauer, Fabrizio Catanese, and Yujiro Kawamata to have had the opportunity to show some of their results in the workshop "Differential, Algebraic and Topological Methods in Complex Algebraic Geometry" Grand Hotel San Michele, Cetraro, September 6-15, 2018. 
\end{ackn}

\section{Adjoint quadrics}

We recall some of the results of \cite{PZ}. See also \cite{RZ}.

\subsection{The Adjoint Theorem}
\label{sezione1}
\subsubsection{The Gauss-type homomorphism}
Let $X$ be a compact complex
smooth variety of dimension $m$ and let $\mathcal{F}$ be a locally free sheaf of rank $n$. 
Fix an extension class $\xi\in \text{Ext}^1(\mathcal{F},\mathcal{O}_X)$ associated to the exact sequence:
\begin{equation}\label{unouno}
0\to\mathcal{O}_X\stackrel{d\epsilon}{\longrightarrow} \mathcal{E}\stackrel{\rho_1}{\longrightarrow} \mathcal{F}\to 0.
\end{equation}
By the Koszul resolution associated to the section $d\epsilon\in H^0(X,\sE)$ and by the isomorphisms 
\begin{equation*}
\text{Ext}^1(\sF,\sO_X)\cong \text{Ext}^1(\bigwedge^i\sF,\bigwedge^{i-1}\sF)
\end{equation*} we see that the coboundary homomorphisms
\begin{equation*}
\partial_\xi^i\colon H^0(X,\bigwedge^i\mathcal{F})\to H^1(X,\bigwedge^{i-1}\mathcal{F})
\end{equation*} are computed by the cup product with $\xi$, $i=1,\dots,n$.

Denote by $H^n_{d\epsilon}\colon \det\mathcal{E}\to \det\mathcal{F}$ the natural isomorphism and by $\Lambda^{n+1}$ the natural map
\begin{equation}
\label{lambdan+1}
\Lambda^{n+1}\colon \bigwedge^{n+1}H^0(X,\mathcal{E})\to H^0(X,\det\mathcal{E}).
\end{equation} By composition we define a Gauss-type homomorphism:
\begin{equation}
\Lambda:=H^{n}_{d\epsilon}\circ \Lambda^{n+1}\colon \bigwedge^{n+1}H^0(X,\mathcal{E})\to H^0(X,\det\mathcal{F}).
\end{equation}

\subsubsection{Adjoint forms} Let $W\subset \Ker(\partial_\xi^1)\subset H^0(X,\mathcal{F})$ be a vector subspace of
dimension $n+1$ and let $\mathcal{B}:=\{\eta_1,\ldots,\eta_{n+1}\}$
be a basis of $W$. By definition we can take liftings 
$s_1,\dots,s_{n+1}\in H^0(X,\mathcal{E})$ such that $\rho_1(s_i)=\eta_i$, $i=1,...,n+1$.

\begin{defn}
The section
\begin{equation*}
\omega_{\xi,W,\mathcal{B}}:=\Lambda(s_1\wedge\ldots\wedge s_{n+1})\in H^0(X,\det\mathcal{F}).
\end{equation*} 
is called \emph{an adjoint form} of $\xi,W,\mathcal{B}$.
\end{defn}

If we consider the natural map
\begin{equation*}
\lambda^n\colon \bigwedge^{n}H^0(X,\mathcal{F})\to H^0(X,\det\mathcal{F}),
\end{equation*} we can define the subspace $\lambda^n W\subset H^0(X,\det\mathcal{F})$ generated by
\begin{equation*}
\omega_i:=\lambda^n(\eta_1\wedge\ldots\wedge\widehat{\eta_i}\wedge\ldots\wedge\eta_{n+1})
\end{equation*} for $i=1,\dots,n+1$.

\begin{defn}
The class
\begin{equation*}
[\omega_{\xi,W,\mathcal{B}}]\in \frac{H^0(X,\det\mathcal{F})}{\lambda^nW}
\end{equation*} is called \emph{the Massey product} of $W$ along $\xi$.
\end{defn}
\medskip

In the literature $[\omega_{\xi,W,\mathcal{B}}]$ is also called the adjoint image of $W$ by $\xi$. For the main properties of Massey products in our context see \cite{CP}, \cite{PZ}, \cite{RZ}.

\begin{defn}\label{fissato}
If $\lambda^nW$ is nontrivial we denote by $|\lambda^n W|\subset \mP(H^0(X,\det\mathcal{F}))$ the induced sublinear system. We call $D_W$ the fixed divisor of this linear system and $Z_W$ the base locus of its moving part $|M_W|\subset\mP(H^0(X,\det\mathcal{F}(-D_W)))$.
\end{defn}

From the natural map
$\epsilon_{D_{W}}\colon \mathcal{F}(-D_W)\to\mathcal{F}$ we have the induced homomorphism in cohomology:
\begin{equation*}
H^1(X,\mathcal{F}^\vee)\stackrel{\epsilon_{D_{W}}}{\longrightarrow} H^1(X,\mathcal{F}^\vee(D_W)).
\end{equation*} We set
$$\xi_{D_W}:=\epsilon_{D_{W}}(\xi).$$

\begin{defn} We say that $\xi\in H^1(X,\mathcal{F}^\vee)$ {\it{is supported on $D_W$ }}if  $\xi_{D_W}=0$
\end{defn}

In \cite[Theorem 1.5.1]{PZ} we have shown:
\begin{thm}[Adjoint Theorem]
\label{teoremaaggiunta}
Let $X$ be a compact $m$-dimensional complex smooth variety. 
Let $\mathcal{F}$ be a rank $n$ locally free sheaf on $X$ and $\xi\in 
H^1(X,\mathcal{F}^\vee)$ the extension class of the exact sequence
(\ref{unouno}). Let $W$ be a $n+1$-dimensional subspace of $\Ker(\partial_\xi^1)\subset 
H^0(X,\mathcal{F})$ and $\omega$ one of its adjoint forms.
If the Massey product $[\omega]=0$ then $\xi$ is supported on $D_W$.
\end{thm}

\subsection{The notion of Adjoint quadric}
We denote by  $\lambda^{n}H^0(X,\sF)$ the image of $$\lambda^n\colon \bigwedge^n  H^0(X,\sF)\to  H^0(X,\det\sF)$$ and we consider the linear subsystem
$\mP(\lambda^{n}H^0(X,\sF))$ of $|\det\sF|$. Denote by $D_\sF$ its fixed 
component and by $|M_{\sF}|$ its associated mobile linear system. 
Moreover we denote $D_{\det\sF}$, $M_{\det\sF}$ respectively 
the fixed and the movable part of $|\det\sF|$; that is:
$|\det\sF|=D_{\det\sF}+|M_{\det\sF}|$. 

\medskip

Take $W=\langle \eta_1,\ldots,\eta_{n+1}\rangle$ and $\omega_i$, $i=1,...,n+1$ as above and let $\omega\in H^0(X, \det\sF)$ be 
a $\xi$-adjoint of $W\subset  H^0(X,\sF)$. Let
$$
\mu_{\det\sF}\colon \colon {\rm{Sym}}^2(H^0(X, \det\sF))\to H^0(X, \det\sF^{\otimes 2})
$$ be the natural multiplication homomorphism. The basic definition of this paper is:
\begin{defn}\label{quadricheaggiunte}
An $\omega$-\emph{adjoint quadric} is an element $Q\in {\rm{Sym}}^2(H^0(X, \det\sF))$ such that
\begin{enumerate}
\item    $Q:=\omega\odot \omega-\sum_{i=1}^{n+1} \omega_i\odot L_i$ for some $L_i\in H^0(X,\det\sF)$, $i=1,...,n+1$;
 \item $\mu_{\det\sF}(Q)=0$.
 \end{enumerate}
\end{defn}
The condition $(2)$ of the above Definition means that $Q$ gives an element of ${\rm{Sym}}^{2}H^0(X, \det\sF)$ which vanishes on the schematic image $\phi_{|M_{\det\sF} |}(X)$.
The study of $\omega$-adjoint quadrics is useful to find extension classes supported on a divisor.
\begin{thm}\label{torelloo1} Let $X$ be a compact complex smooth variety. Let $\sF$ be a 
		locally free sheaf of rank 
    $n$ such that $h^{0}(X, \sF)\geq n+1$. Let $\xi\in 
    H^1(X,\sF^\vee)$ and let $Y$ be the schematic 
    image of $\phi_{|M_{\det\sF} |}\colon X\dashrightarrow 
    \mP(H^{0}(X,\det\sF)^{\vee})$. 
    If $\xi$ is such that 
    $\partial_{\xi}^{n}(\omega)=0$, where $\omega$ is an adjoint form associated
    to an $n+1$-dimensional subspace $W\subset\Ker\partial_\xi\subset 
    H^0(X,\sF)$, then $[\omega]=0$, providing that there are no 
    $\omega$-adjoint quadrics.
\end{thm}
\begin{proof}
Let $\sB=\{\eta_{1},\ldots, \eta_{n+1}\}$ be 
a basis of $W$. Set 
$\omega_i$ for $i=1,\dots,n+1$ as above and denote by $\tilde{\omega}_i\in 
H^0(\det\sF(-D_W)\otimes \sI_{Z_W})$ the corresponding sections 
via $0\to H^0(X, \det\sF(-D_W)\otimes \sI_{Z_W})\to H^0(X,\det\sF)$.
Recall that $\lambda^nW:=\left\langle 
\omega_1,\dots,\omega_{n+1}\right\rangle\subset 
H^0(X,\det\mathcal{F})$ is the vector space  generated by the sections $\omega_i$. The standard evaluation
     map $\bigwedge^{n}W\otimes\sO_{X}\to 
     \det\sF(-D_W)\otimes \sI_{Z_{W}}$ given by 
     $\tilde{\omega}_1,\dots,\tilde{\omega}_{n+1}$ gives the following exact sequence
     
\begin{equation}\label{genera1}
\xymatrix { 0\ar[r]&\sK \ar[r]&
\bigwedge^{n}W\otimes\mathcal{O}_X\ar[r] & \det\mathcal{F}(-D_W)\otimes \sI_{Z_W}\ar[r]&0}
\end{equation} which is associated to a class $\xi'\in \text{Ext}^1(\det\sF(-D_W)\otimes \sI_{Z_W},\sK)$.
The sequence (\ref{genera1}) fits into the following commutative diagram 
\begin{equation}
\xymatrix { 0\ar[r]&\sK \ar[r]& \bigwedge^{n}W\otimes\mathcal{O}_X\ar[r] & \det\mathcal{F}(-D_W)\otimes \sI_{Z_W}\ar[r]&0\\
0\ar[r]&\sF^\vee\ar[r]\ar[u]&\sE^\vee\ar[r]\ar[u]^{f}&\mathcal{O}_X\ar[u]^{g}\ar[r]&0,
}
\end{equation} where $f$ is the map given by the contraction by the sections $(-1)^{n+1-i}s_i$, for $i=1,\dots,n+1$, and $g$ is given by the global section $\sigma \in H^0(X,\det\sF(-D_W)\otimes \sI_{Z_W})$ corresponding to the adjoint form $\omega$.
We have the standard factorization
\begin{equation}
\xymatrix { 0\ar[r]&\sK \ar[r]& \bigwedge^{n}W\otimes\mathcal{O}_X\ar[r] & \det\mathcal{F}(-D_W)\otimes \sI_{Z_W}\ar[r]&0\\
0\ar[r]&\sK \ar[r]\ar@{=}[u]&\sL\ar[r]\ar[u]&\sO_X\ar[u]^{g}\ar@{=}[d]\ar[r]&0\\
0\ar[r]&\sF^\vee\ar[r]\ar[u]&\sE^\vee\ar[r]\ar[u]&\mathcal{O}_X\ar[r]&0
}
\end{equation} where the sequence in the middle is associated to the class
$\xi''\in H^1(X,\sK)$ which is the image of $\xi\in H^1(X,\sF^\vee)$ through the map $H^1(X,\sF^\vee)\to H^1(X,\sK)$.
In particular we obtain the commutative square:
\begin{equation}
\xymatrix {H^0(X,\det \sF(-D_W)\otimes \sI_{Z_W})\ar[r]&H^1(X,\sK)\ar@{=}[d]\\
H^0(X,\sO_X)\ar[u]\ar[r]& H^1(X,\sK).
}
\end{equation} By commutativity, we immediately have that the image of $\sigma \in H^0(X,\det\sF(-D_W)\otimes \sI_{Z_W})$ through
the coboundary map $H^0(X,\det \sF(-D_W)\otimes \sI_{Z_W})\to H^1(X,\sK)$ is $\xi''$.
Tensoring by $\det \sF$, the map $\sF^\vee\to\sK$ gives 
\begin{equation}
\xymatrix { \sF^\vee\otimes\det\sF \ar@{=}[d]\ar[r] &\sK\otimes\det\sF\\
\bigwedge^{n-1}\sF\ar[ur]^{\Gamma}
}
\end{equation} and, since $\xi\cdot\omega\in 
H^1(X,\sF^\vee\otimes\det\sF)$ is sent to $\xi''\cdot\omega\in H^1(X,\sK\otimes\det\sF)$, we have that
\begin{equation}
H^{1}(\Gamma)(\xi\cup \omega)=\xi''\cdot\omega,
\end{equation} where $\xi\cup \omega$ is the cup product.

By hypothesis $\partial_\xi^n(\omega)=\xi\cup \omega=0\in H^1(X,\bigwedge^{n-1}\sF)$, so also 
$\xi''\cdot \omega=0\in H^1(X,\sK\otimes\det\sF)$, hence the global section $\sigma\cdot\omega\in H^0(X,\det \sF(-D_W)\otimes \sI_{Z_W}\otimes\det\sF)$ is in the kernel of the coboundary map $H^{0}(X,\det \sF(-D_W)\otimes \sI_{Z_W}\otimes\det\sF)\to H^{1}(X, 
\sK\otimes\det\sF)$ associated to the sequence

\begin{equation}
\xymatrix { 0\ar[r]&\sK\otimes\det\sF \ar[r]&
\bigwedge^{n}W\otimes\det\sF\ar[r] & \det\mathcal{F}(-D_W)\otimes \sI_{Z_W}\otimes\det\sF\ar[r]&0.}
\end{equation}
This occurs iff there exist 
$L^{\sigma}_{i}\in 
H^{0}(X,\det\sF)$, $i=1,\ldots, n+1$ such that
\begin{equation}
\sigma\cdot\omega=\sum_{i=1}^{n+1}\tilde{\omega}_{i}\cdot L^{\sigma}_{i}.
\end{equation}
This relation gives the following relation in $H^{0}(X,\det
\sF^{\otimes2})$:
\begin{equation}\label{quadra}
\omega\cdot\omega=\sum_{i=1}^{n+1}L^{\sigma}_{i}\cdot\omega_{i}.
\end{equation}
Then the equation (\ref{quadra}) gives an adjoint 
quadric. By contradiction the claim follows.
\end{proof}

\begin{cor}
\label{torelloooo} In the hypothesis of  Theorem \ref{torelloo1} it holds that $\xi$ is supported on $D_{W}$; that is, $\xi_{D_{W}}$ is trivial. Moreover if we further assume that $W$ is generic inside $H^{0}(X, \sF)$ it follows that 
$\xi$ is supported on $D_{\sF}$.
\end{cor}
\begin{proof}
The first claim follows by Theorem 
\ref{torelloo1} and by Theorem \ref{teoremaaggiunta}. To show the second claim, we recall that by \cite[Proposition 3.1.6]{PZ} $D_{\sF}=D_W$ since $W$
is a generic $n+1$-dimensional subspace of $H^0(X,\sF)$. Then the claim follows.
\end{proof}

\section{Nori Families} 
We apply the notion of adjoint quadrics to the case where $\sF$ is the cotangent sheaf $\Omega^{1}_{X}$ of a smooth variety. We stress that we want to find conditions on a family $\pi\colon\sX\rightarrow B$ which ensure that the fibers are in the same birational class.
\subsection{A notion of equivalence among families of Albanese type}
The notion of Albanese type family behaves well under base change and we can introduce a notion of equivalence for this kind of families. Consider a family of Albanese type $\Phi\colon \sX{\rightarrow}\sA$ as in the Introduction.

\subsubsection{Translation equivalence} If $s\colon B\to\sA$ is a section of
$p:\sA\rightarrow B$, we define the translated family
$\Phi_{s}:\sX\rightarrow
\sA$ of $\Phi$ by the formula:
$$\Phi_{s}(x)=\Phi(x)+ s(\pi(x)).$$
Notice that $\Phi_{s}:\sX\rightarrow
\sA$ is a family of Albanese type.
Two
families
$\Phi$ and $\Psi$ over $B$
are said to be {\it{translation equivalent}}
if there exists a section $\sigma$ of $p\colon \sA\rightarrow B$ such that
the {\it{images}} of
$\Phi_{\sigma}$
and $\Psi$ (fiberwise) coincide.

 We recall also the following definition given in 
\cite[definition 1.1.2]{PZ}:

\begin{defn}\label{panna}  Two families of Albanese type
$\Phi\colon \sX{\rightarrow}\sA$,
$\Phi'\colon \sX^{'} {\rightarrow}\sA^{'}$ over, respectively, $B$ and $B'$ will be said
{\it{locally translation equivalent}}, if there exist an open set $U\subset B$ an open set $U'\subset B'$ and a biregular map $\mu \colon U'\to U:=\mu(U')\subset B$ such that the pull-back families 
$\mu^{\ast}( \Phi_{U})$ and $\Phi^{'}_{U^{'}}$ are {\rm{translation equivalent}} where $U$, $U'$ are dense with respect to the classical topology on $B$ respectively $B'$. We will say that $\Phi$ is {\it{trivial}} if $\sX=X\times B$,
$\sA=A\times B$ and $\pi_{A}(\Phi({X_{b}}))=\pi_{A}(\Phi(X_{b_{0}}))$ for all $b$ where $\pi_{A}\colon A\times B\to A$ is the natural projection.
\end{defn}

\medskip

We will use the following:
\begin{prop}\label{pomodoro}
An Albanese type family $\Phi\colon\sX\to\sA$ is locally translation equivalent to a trivial family if
and only if the fibers $X_{b}$ are birationally equivalent.
\end{prop}
\begin{proof} See \cite[Proposition 1.1.3]{PZ}.\end{proof}

\subsection{Liftability assumptions}
The following conditions are natural in order to find  families locally translation equivalent to trivial families.

\begin{defn}\label{torellifamily} We say that a family $\pi\colon \sX {\rightarrow}B$ of relative dimension $n$ 
    satisfies extremal liftability conditions over a $1$-dimensional variety $B$ if
\begin{enumerate}
\item $ H^0({\sX},\Omega_{{\sX}}^1)
\twoheadrightarrow H^0(X_{b},\Omega_{X_{b}}^1)$;
\item $H^0({\sX},\Omega_{{\sX}}^n)
\twoheadrightarrow H^0(X_{b},\Omega_{X_{b}}^n)$
\end{enumerate}
where the symbol $\twoheadrightarrow$ means that the homomorphism is surjective. 
\end{defn}

\medskip

The above definition says that all the $1$-forms and all the $n$-forms of the fiber $X_b$ are obtained by restriction of forms defined on the family $\sX$. Comparing the two conditions with the hypotheses of Theorem \ref{torelloo1} we see that they ensure that $\partial^{1}_{\xi_b}=0$ and $\partial_{\xi_{b}}^n=0$, where $\xi_b\in H^1(X_b, \Theta_{X_{b}})$ is an infinitesimal deformation in the image of the Kodaira-Spencer map associated to $\pi\colon \sX {\rightarrow}B$.

\begin{prop}\label{insalata}
Let $\Phi\colon\sX\to\sA$ be an Albanese type family such that for 
every $b\in B$ it holds that $ H^0({\sX},\Omega_{{\sX}}^1)
\twoheadrightarrow H^0(X_{b},\Omega_{X_{b}}^1)$. Then up to shrinking 
$B$ the fibers of 
$p\colon\sA\to B$ are isomorphic.
\end{prop}
\begin{proof} Let $\mu_b\in {\rm{Ext}}^1(\Omega^{1}_{A_b},\sO_{A_b})$
    be the class given by the family $p\colon\sA\to B$,
 that is the class of the following extension:
$$
0\to \sO_{A_{b}}\to\Omega^1_{\sA|A_{b}}\to\Omega^1_{A_{b}}\to 0.
$$\noindent
Now $\phi_b^*\sO_{A_{b}}=\sO_{X_b}$ and the map $\phi_b^*\sO_{A_{b}}\to\phi_b^*\Omega^1_{\sA|A_{b}}$ is generically injective, hence it is injective because otherwise the kernel would be a torsion subsheaf of $\sO_{X_b}$. Thus we have the following exact sequence
$$
0\to \phi_b^*\sO_{A_{b}}\to\phi_b^*\Omega^1_{\sA|A_{b}}\to\phi_b^*\Omega^1_{A_{b}}\to 0
$$ which fits into the following diagram
\begin{equation*}
\xymatrix { 0\ar[r]&\phi_b^*\sO_{A_{b}} \ar[r]\ar@{=}[d]&
\phi_b^*\Omega^1_{\sA|A_{b}}\ar[r]\ar[d] & \phi_b^*\Omega^1_{A_{b}}\ar[d]\ar[r]&0\\
0\ar[r]&\sO_{X_b} \ar[r]&
\Omega^1_{\sX|X_{b}}\ar[r] & \Omega^1_{X_{b}}\ar[r]&0.
}
\end{equation*}
In cohomology we have
\begin{equation*}
\xymatrix {H^0(X_b,\phi_b^*\Omega^1_{A_{b}})\ar[d]\ar[r]&H^1(X_b, \sO_{X_b})\ar@{=}[d]\\
H^0(X_b,\Omega^1_{X_{b}})\ar[r]&H^1(X_b, \sO_{X_b})
}
\end{equation*} so, by commutativity and by the hypothesis $ H^0({\sX},\Omega_{{\sX}}^1)
\twoheadrightarrow H^0(X_{b},\Omega_{X_{b}}^1)$, we immediately obtain $ H^0(X_b,\phi_b^*\Omega^1_{\sA|A_{b}})
\twoheadrightarrow H^0(X_{b},\phi_b^*\Omega^1_{A_{b}})$ and hence the coboundary $\partial_{\mu_{b}}\colon H^0(A_b ,\Omega^1_{A_{b}})\to H^1(A_b, \sO_{A_{b}})$ is trivial. Then by cf. \cite [Page  78]{CP} we conclude.
\end{proof}

\subsection{Nori Families}
Let $\Phi:\sX\rightarrow\sA$ be an Albanese type family over the unitary disc
$\Delta$. From $\Phi(\sX)\hookrightarrow\sA$ we obtain the basic cycle
$[\sZ]=[\sX]^{+}-[\sX]^{-}$ as in (\ref{cnstccl}). First we see that $[\sZ]\in
Z^{a-n}_{h}(\sA/B)$. To the normal function defined by $\sZ$ it is associated its infinitesimal invariant $\delta_{\sZ}$; see: cf.\cite{Vo}.

\begin{defn}
An Albanese type family $\Phi$ is called Nori trivial if the
infinitesimal invariant
$\delta_{\sZ}$ induced by the cycle $[\sZ]$ is zero for the generic
$b\in B$ (hence for all $b$).
\end{defn}
 
 \noindent {\bf{Transversality.}} Fix  $s_{1},\dots ,s_{n+1}\in H^{0}(\sA,\Omega^{1}_{\sA_{}}\,)$ such
that
$s_{1}\wedge\dots \wedge s_{n+1}$ induces, by fiber restriction, a non trivial form $\Omega\in
H^{0}(A_{b},\Omega^{1+n}_{\sA\mid A_{b}}\,).$ Let
$\xi_b\in H^{1}(X_{b},T_{X_{b}})$
be an infinitesimal deformation given by the Kodaira-Spencer map. Let
$$r:\Phi_{b}^{\ast}\Omega^{1}_{\sA}\rightarrow \Omega^{1}_{X_b}$$ be the
restriction map and set
$\eta_{i}=r(\Phi_{b}^{\ast}(s_{i}))$, $i=1,...,n+1.$ In our case the set $\sB=\{\eta_{i}\}_{i=1}^{n+1}$ is a basis of a vector space $W\subset
H^{0}(X_{b},\Omega^{1}_{X}\,)$. Suppose that
$W\subset{\rm{Ker}}\partial_{\xi_b}.$ Then the element
$\Phi_{b}^{\ast}(s_{1}\wedge\dots \wedge s_{n+1})$ gives precisely an adjoint form $\omega_{_{\xi_b,W,\sB}}$ once it is restricted to $X_b$. Now consider the factorisation $\phi_b\colon X_b\to A_b$ induced by the morphism $\Phi\circ j_b\colon X_b\to\sA$ where $j_b\colon X_b\to\sX$ is the natural inclusion. In \cite[Theorem 5.2.5]{PZ} it is proved:
\begin{thm}\label{trasversalita}{\bf{Transversality Criteria.}}
If $\delta_{\sZ}(b)=0$ then for every
$\sigma\in H^{0}(A_{b},\Omega^{n}_{A_{b}}\,)$ it holds:
$$\int_{X}\omega_{_{\xi_b,W,\sB}}
\wedge{\overline{\phi^{\star}_b\sigma}}=0.$$
\end{thm}

\subsection{Proof of the Main theorem}

\begin{proof}  
 By Proposition \ref{insalata} we can assume that $p\colon\sA\to B$ is trivial, that is $\sA\simeq A\times B$ and $p\colon\sA\to B$ is the first projection. Up to base change, the Albanese family ${\rm{alb}}(\sX)\colon \sX\to  {\rm{Alb}}(\sX)$ exists and by Proposition \ref{pomodoro}, our claim is equivalent to show that  the Albanese family ${\rm{alb}} (\sX) \colon \sX\to  {\rm{Alb}} (\sX)$
 is locally translation equivalent to the trivial family. Hence it is not restrictive to assume that ${\rm{Alb}}(\sX)=A\times B$ too. In particular we can restrict to consider only  the case where  ${\rm{Alb}}(X_b)=A$ for every $b\in B$.

 Denote by $\xi_b\in H^1(X_b,\Theta_{X_b})$ a class associated to an infinitesimal deformation of $X_b$ induced
by the fibration $\pi\colon\sX\rightarrow B$. We know that $q\geq n+1$ where $q={\rm{dim}}_\mathbb C A$. 
Let $\sB:=\{dz_1,\ldots,dz_{n+1}\}$ be a 
basis of an $n+1$-dimensional generic subspace $W$ of $H^0( A,\Omega^{1}_{A})$, (if $q=n+1$ we can take $H^0( A,\Omega^{1}_{A})=W$). 
For every $b\in B$ let $\eta_i(b):= {\rm{alb}}(X_b)^{\star}dz_i$, $i=1,\ldots, n+1$. By standard
theory of the Albanese morphism it holds that $\sB_b:=\{ 
\eta_1(b),\ldots, \eta_{n+1}(b)\}$ is a basis of the pull-back $W_{b}$ 
of $W$ inside $H^0(X_b,\Omega_{X_{b}}^1)$. 
Let $$\omega_i (b):= 
\lambda^{n}(\eta_1(b)\wedge\ldots\wedge\eta_{i-1}(b)\wedge{\widehat{\eta_i(b)}}\wedge\ldots\wedge\eta_{n+1}(b))$$ for $i=1,\ldots, n+1$. Note that if 
$\omega'_i:=dz_1\wedge \ldots \wedge dz_{i-1} \wedge {\widehat{dz_i}} \wedge\ldots\wedge dz_{n+1}$ then $\omega_i (b):=\phi_{b}^{\star}\omega_i'$, $i=1,...,n$.
Since $\Phi\colon\sX\to\sA$ is a family of Albanese type, 
${\rm{dim}}\lambda^nW_{b}\geq 1$, (actually if $q>n+1$ by \cite[Theorem 1.3.3]{PZ} it follows
that $\lambda^nW_{b}$ has dimension $n+1$),
and we can write:
$\lambda^nW_{b}=\langle \omega_1(b),\ldots,\omega_{n+1}(b)\rangle$. 

By extremal liftability assumptions we can form the Massey class for every $[W]\in\mathbb G(n+1,q)$ where we denote by $\mathbb G(n+1,q)$ the Grassmannian of $n+1$ dimensional subspaces of $H^0(X_b,\Omega^1_{X_b})$. 

Consider the following diagram in Dolbeaut's cohomology:

\begin{equation}\label{moltiplica}
\xymatrix{
H^0(X_b, \phi_{b}^{\star}\Omega^n_{A}\otimes \omega_{X_b})\ar^{\mu\circ j}[r] \ar_{j}[d]& H^0(X_b,\omega_{X_b} \otimes \omega_{X_b} )\\
H^{n,0}(X_b)\otimes H^{n,0}(X_b)\ar_{\mu}[ur]&
}
\end{equation}

\noindent
where by the identification $H^0(X_b,\phi_{b}^{\star}(\Omega^n_{A})\otimes \omega_{X_b})=H^{n,0}(A)\otimes H^{n,0}(X_b)$ it follows that
$$j:=\phi_b^{\star}\otimes {\rm{id}}\colon H^{n,0}(A)\otimes H^{n,0}(X_b)\to H^{n,0}(X_b)\otimes H^{n,0}(X_b).$$
 We set $V:= {\rm{Im}}\phi_b^{\star}\subset  H^{n,0}(X_b)$. We consider the symmetric dual of the multiplication map. By the above discussion we can write:
\begin{equation}\label{somsom}
\mu^{\vee}=\mu_{\overline{V}}\oplus\rho\colon H^n(X,-K_{X})\longrightarrow {\overline{V}}\odot H^{0,n}(X_b)\oplus {\rm{Sym}}^2({\rm{Ann}}(V))
\end{equation}
The symmetric dual of the diagram (\ref{moltiplica}) is:

\begin{equation}\label{duediagramma}
\xymatrix{
 {\overline{V}} \odot H^{0,n}(X_b)\oplus {\rm{Sym}}^2({\rm{Ann}}(V))\ar_{j^\vee}[d]    & H^n(X_b,-K_X)\ar_-{\mu_{\overline{V}}\oplus\rho}[l]\ar_-{\hat{\rho}}[ld]\\
H^{q-n,q}(A)\odot H^{0,n}(X_b)&
}
\end{equation}

where we have used the identifications
\begin{equation}\label{diagrammabase}
\begin{split}H^n(X_b, \phi_{b}^{\star}(\bigwedge^n\Theta_{A}))=H^n(X_b,\phi_{b}^{\star}(\bigwedge^nH^0(A,\Omega^1_{A})^\vee\otimes\sO_{A}))=\bigwedge^nH^0(A,\Omega^1_{A})^\vee\otimes H^n(X_b,\sO_{X_b})=\\=H^0(A,\Omega^n_{A})^\vee\otimes H^n(X_b,\sO_{X_b})=(H^{n,0}(A))^\vee\otimes H^{0,n}(X_b)\end{split}
\end{equation}
and 
$$
(H^{n,0}(A))^\vee=H^{q-n,q}(A).
$$
We stress that since we can write 
$$
{\rm{Sym}}^2H^{n,0}(X_b)=V\odot H^{n,0}(X_b)\oplus {\rm{Sym}}^2({\rm{Ann}}({\overline{V}}))
$$
we also have
$$
{\rm{Sym}}^2H^{0,n}(X_b)={\overline{V}}\odot H^{0,n}(X_b)\oplus {\rm{Sym}}^2({\rm{Ann}}({{V}}))
$$

By hypothesis the restriction 
$$\nu_{X_b}=\mu_{|{\rm{Sym}}^2({\rm{Ann}}({\overline{V}}) )} \colon
{\rm{Sym}}^2({\rm{Ann}}({\overline{V}}) )\to H^0(X_b,\omega_{X_{b}}^{\otimes 2})$$ is injective hence  ${\rm{Ker}}(\mu)\cap {\rm{Sym}}^2({\rm{Ann}}({\overline{V}}) )=\{0\}$.

Now assume that for the generic $W_b$, the generic adjoint form $\omega$ has an adjoint quadric $$Q:=\omega\odot\omega -\sum_{i=1}^{n+1}  \omega_i(b)\odot L_i\in {\rm{Sym}}^2H^{n,0}(X_b).$$ By Definition \ref{quadricheaggiunte} $Q$ is in ${\rm{Ker}}(\mu)$. Hence $Q$ vanishes on 
$ ({\rm{Sym}}^2({\rm{Ann}}({\overline{V}}) ))^\vee$. That is $Q$ vanishes on
${\rm{Sym}}^2({\rm{Ann}}({{V}}))=({\rm{Sym}}^2({\rm{Ann}}({\overline{V}}) ))^{\vee}$.

Now consider any $\alpha\in  H^n(X,-K_{X})$. By Equation \ref{somsom} it holds that $\mu^\vee(\alpha)= \mu_{\overline{V}}(\alpha)+\rho(\alpha)$ where 
$\rho(\alpha)\in  {\rm{Sym}}^2({\rm{Ann}}({{V}}))$.
It holds
$$
\langle\omega^2,\alpha\rangle=\langle \omega\odot\omega,\mu^\vee(\alpha)\rangle=\langle \omega\odot\omega, \mu_{\overline{V}}(\alpha)+\rho(\alpha)\rangle
$$
By the Transversality Theorem $\langle \omega\odot\omega, \mu_{\overline{V}}(\alpha)\rangle=0$. We claim that also 
$\langle \omega\odot\omega,\rho(\alpha)\rangle=0$. Indeed notice that since $Q \in {\rm{ Ker}} (\mu)$ it holds that
 $\langle Q,\mu^\vee(\alpha)\rangle=0$ but also that $\langle Q,\rho(\alpha)\rangle=0$ since $\rho(\alpha)\in {\rm{Sym}}^2({\rm{Ann}}({{V}}))$. This means
$$
\langle\omega\odot\omega ,\rho(\alpha)\rangle=\langle\sum_{i=1}^{n+1}\omega_i(b)\odot L_i,\rho(\alpha)\rangle.
$$
We have shown that $\rho(\alpha)\in {\rm{Sym}}^2({\rm{Ann}}({{V}}))$ and by definition $\omega_i\in V$ then$\langle\omega_i\odot L_i,\rho(\alpha)\rangle=0$ for every $i=1,..., n+1$. 
Therefore, assuming that for the generic $W_b$, the generic adjoint form $\omega$ has an adjoint quadric, we conclude thet $\omega^2$ is the trivial functional on $H^n(X,-K_X)$. By \cite[Theorem 1.5.3]{PZ} the claim follows easily. 

On the other hand if there are no $\omega$-adjoint quadrics then by Theorem \ref{torelloo1} we have that $[\omega]=0$ and again by \cite[Theorem 1.5.3]{PZ} we conclude.
\end{proof}

\begin{cor}\label{caso} Let $\Phi:\sX\rightarrow\sA$ be a family of Albanese type where the generic fiber
$A_{b}$ of $p:\sA\rightarrow B$ is a generic Abelian variety of dimension $\geq 4$ and where the map $\nu_{X_b}\colon {\rm{Sym}}^{2}{\overline{{\rm{Ann}}(V_b)}}\to H^0(X,\omega_{X_b}^{\otimes 2})$ is injective for a general $b\in B$. If it satisfies the liftability conditions then it has birational fibers. In particular a family of smooth varieties of general type all contained inside a generic Abelian variety of dimension $\geq 4$ has birational fibers if it satisfies the liftability assumptions and the injectivity property.
\end{cor}
\begin{proof} By \cite[Proposition 6.2.2]{PZ} we know that $\Phi:\sX\rightarrow\sA$ is equivalent to a Nori trivial family. By the Main Theorem the claim follows.
\end{proof}

\begin{rmk} We point out the reader that since the Ceresa cycle of an hyperelliptic curve $C$ is trivial, the Main Theorem implies that a family $\pi\colon\sC\to B$ of hyperelliptic curves satisfying  liftability assumptions is a locally trivial family.
\end{rmk}
\section{Families of divisors of a polarized Abelian variety}
\subsection{Theta functions}
Let $(\Ab, \AL)$ be a $(d_1 \cdots d_g)$-polarized abelian variety, where $\Ab$ is a complex torus defined as a quotient of a vector space $\VSA$ of rank $g$ by a lattice $\LATT$, and $\AL$ an ample line bundle on it. The algebraic equivalence class of line bundles of $\mathcal{L}$ is defined by a non-degenerate hermitian bilinear form $H$ on $\VSA$, whose imaginary part $E$ is a bilinear form integer-valued on $\LATT$. Since we are interested in the algebraic equivalence class of line bundles on $\Ab$ defined by $\mathcal{L}$, we may assume the characteristic of $\mathcal{L}$ to be $0$. We also recall that $\mathcal{L}$ determines an isogeny
\begin{equation*}
\phi_{\mathcal{L}} \colon \Ab \longrightarrow {\rm{Pic}}^0(\Ab)
\end{equation*}
which is defined as follows:
\begin{equation*}
\phi_{\mathcal{L}}(z):=t^*_z(\mathcal{L}) \otimes \mathcal{L}^{-1}
\end{equation*}
A decomposition of $V$ for $\mathcal{L}$ is a decomposition of $\VSA = \VSA_1 \oplus \VSA_2$ into real vector spaces of rank $g$ which induce a decomposition for $\LATT = \LATT_1 \oplus \LATT_2$ into $E$-isotropic free $\mZ$-modules of rank $g$. 
Such a decomposition of $V$ for $\mathcal{L}$ induces moreover a decomposition of the lattice
\begin{equation}
\LATT(\mathcal{L}) := \{ v \in V \ : \ t^*_z\mathcal{L} \cong \mathcal{L}, [v]=z \}
\end{equation}
into $E$-isotropic free $\mZ$-modules of rank $g$, which we respectively denote by $\LATT(\mathcal{L})_1$ and $\LATT(\mathcal{L})_2$. The latter decomposition naturally induces a decomposition of the kernel of $\phi_{\mathcal{L}}$, which we denote by $K$. It is known, see c.f. \cite[Theorem 2.7 p.55]{BL}, that 
\begin{equation}
\{ \ThfC{x}{\mathcal{L}} \ \ : \ x \in K_1  \}
\end{equation}
is a basis for $H^0(\Ab, \AL)$, where
\begin{align*}
 \ThfC{0}{\mathcal{L}}(z) := \sum_{\lambda \in \LATT_1}e^{\pi(H-B)(z, \lambda)-\frac{\pi}{2}(H-B)(\lambda, \lambda)} 
\end{align*}
Here $B$ denotes the $\mC$-linear extension of $H|_{\VSA_2 \times \VSA_2}$, and for every $x$ in $K_1$
\begin{align*}
 \ThfC{x}{\mathcal{L}}(z) := \psi^{\mathcal{L}}_x(z)^{-1}\ThfC{0}{\AL}(z+x)
\end{align*}
  where $\{\psi_{\lambda}\}_{\lambda}$ is the cocycle in $Z^1(\LATT, \mathcal{O}_{\VSA})$ such that, for every $\lambda$ in the lattice $\LATT$ and $z$ in $\VSA$, we have
\begin{equation*}  
\ThfC{0}{\mathcal{L}}(z+\lambda) = \psi_{\lambda}(z)\ThfC{0}{\mathcal{L}}(z)
\end{equation*}
\begin{prop}
 Let $(\Ab, \mathcal{L})$ be an abelian variety, and $\AD$ be a divisor in the linear system $|\mathcal{L}|$. Then there is a commutative diagram
 \begin{equation} \label{diagDER}
\xymatrix { H^0(\Ab, \Theta_{\Ab}) \ar[r]^{d_0\phi_{\mathcal{L}}}& H^1(\Ab, \mathcal{O}_{\Ab})\\
H^0(\Ab, \Omega_{\Ab}^{g-1})\ar[r]^{|_{\AD}}\ar[u]_{\cong}& H^0(\AD, \omega_{\AD}) \ar[u]^{\ConnHom}}
\end{equation}
where the arrow on right side of diagram \ref{diagDER} is the connecting homomorphism in the long exact cohomology sequence of the fundamental sequence of $\AD$
\begin{equation}\label{funddiagr}
 0 \longrightarrow \mathcal{O}_{\Ab} \longrightarrow \mathcal{O}_{\Ab}(\AD) \longrightarrow \omega_{\AD} \longrightarrow 0
\end{equation}
\end{prop}
\begin{proof}
We assume that $\AD$ is the zero locus of a holomorphic section $s$ of $\mathcal{L}$. Fixed $\omega$ a non-zero $(g,0)$-form on $\Ab$, Recall that $H^0(A, \Theta_{\Ab}) \longrightarrow H^0(\Ab, \Omega_{\Ab}^{g-1})$ sends a derivation $\frac{\partial}{\partial v}$ in $0$ to the unique holomorphic $g-1$-form $w$ such that $dv \wedge w = \omega$. On the other side, the function $\frac{\partial s}{\partial v}$ can be seen by adjunction as a holomorphic section of the canonical bundle of $\AD$, which coincides with the restriction to $\AD$ of the $(g-1)$-form $w$ defined above. On the other side, the connecting homomorphism can be computed by using the fact that there is a canonical isomorphism of cohomology groups sequences
\begin{equation*}
H^p(\LATT, H^0(\VSA, \pi^*(\cdot)) \cong H^p(A, \cdot)
\end{equation*}
where $\pi$ denotes the projection of $\VSA$ onto $\Ab$, and it holds
\begin{equation}\label{cocycleC1}
f\left(\frac{\partial s}{\partial v} \right) = [{\pi H(v,\lambda)}_{\lambda \in \LATT}]
\end{equation}
It remains to compute $d_0\phi_{\mathcal{L}}\left(\frac{\partial}{\partial v} \right)$. Let us consider $S := Spec\left(\bigslant{\mC[\epsilon]}{\epsilon^2}\right)$ the scheme of dual numbers over $\mC$ and $\Ab_S$ the base change.
We have the exact sequence of sheaves
\begin{equation*}
0 \longrightarrow \mathcal{O}_{\Ab} \longrightarrow \mathcal{O}^*_{\Ab_S} \longrightarrow \mathcal{O}^*_{\Ab} \longrightarrow 0
\end{equation*}
Its long cohomology sequence identifies $H^1(\Ab, \mathcal{O}_{\Ab})$ with the kernel of the map ${\rm{Pic}}(\Ab_S) \longrightarrow {\rm{Pic}}(\Ab)$, which to a line bundle on $\Ab_S$ whose transition functions $g_{\alpha\beta}=g'_{\alpha\beta} + \epsilon g''_{\alpha\beta}$ associates the line bundle on $\Ab$ with transition functions $g'_{\alpha\beta}$. Moreover, under the identification $H^1(\Ab, \mathcal{O}_{\Ab}^*) \cong H^1(\Lambda, H^0(\VSA, \mathcal{O}_{\VSA}))$
\begin{equation} \label{CocycleExt}
{\rm{Pic}}(\Ab_S) = H^1(\Lambda, H^0(\VSA,\mathcal{O}_{\VSA})\otimes_{\mC} \mC[\epsilon]) 
\end{equation}
since $\Ab_S$ is defined through a flat base change. Now, for every $z$ on $A$, $\phi_{\mathcal{L}}(z)$ is the line bundle of degree $0$ with cocycles $[\{e^{2 \pi i E(z, \lambda)}\}_{\lambda}]$. Hence, $d_0\phi_{\mathcal{L}}\left(\frac{\partial}{\partial v} \right)$ is the line bundle on $\Ab_S$ whose cocycles, according to identification \ref{CocycleExt}, are precisely
\begin{equation*}
[\{e^{2 \pi i \epsilon E(v, \lambda)}\}_{\lambda}] = [\{1 +  2\pi i E(v,\lambda)\epsilon\}_{\lambda}] \in H^1(\Lambda, H^0(\VSA,\mathcal{O}_{\VSA})\otimes_{\mC} \mC[\epsilon])
\end{equation*}
In conclusion, we have
\begin{equation} \label{cocycleC2}
d_0\phi_{\mathcal{L}}\left(\frac{\partial}{\partial v} \right) = [\{2\pi i E(v,\lambda)\}_{\lambda}] 
\end{equation}
It is now easy to see that the two elements in the cohomology group $H^1(\Lambda, H^0(\VSA,\mathcal{O}_{\VSA}))$ are the same. Indeed, it is enough to show, by the definitions of group cohomology, that there exists a holomorphic function $F$ on $V$ such that, for every $z$ on $V$ and every $\lambda$ on $\LATT$, it holds that
\begin{equation} \label{finalRelCocycles}
\pi H(v,\lambda) = 2\pi i E(v,\lambda) + F(z+\lambda)-F(z)
\end{equation}
But $E$ is defined as the imaginary part of $H$, which is an alternating $\mR$-bilinear form on $\VSA$, and $H$ can be recovered by $E$. Indeed, for every $z$ and $w$ on $\VSA$ it holds:
\begin{equation*}
H(z,w) = iE(z,w)+E(iz,w)
\end{equation*}
In conclusion, with $F(z) := -\pi(i E(v,z) - E(iv, z))$, it is easily seen that $F$ is $\mC$-linear on $\VSA$ and that \ref{finalRelCocycles} holds true.
\end{proof}
\subsection{The multiplication map}
From diagram (\ref{diagDER}) and the long cohomology sequence of (\ref{funddiagr}) it follows easily that 
\begin{align*}
Im(|_{\AD}) &\cong H^1(\Ab, \mathcal{O}_{\Ab}) \cong V \\
\overline{Ann(V)} &\cong Im(H^0(\Ab, \mathcal{L}) \longrightarrow H^0(\AD, \omega_{\AD})).
\end{align*}
Moreover, we have clearly a commutative diagram
 \begin{equation} \label{diagmult1}
\xymatrix { {\rm{Sym}}^{2}{H^0(\Ab, \mathcal{L})} \ar[r]^{\mu}\ar[d]_{|_{\AD}}& H^0(\Ab, \mathcal{L}^2)\ar[d]^{|_{\AD}} \\
{\rm{Sym}}^{2}{\overline{Ann(V)}}\ar[r]^{\nu}& H^0(\AD, \omega_{\AD}^2). }
\end{equation}

Note that ${\rm{Sym}}^{2}{H^0(\Ab, \mathcal{L})}\cong s^{\otimes 2}\mathbb C\oplus s\otimes {\overline{Ann(V)}} \oplus {\rm{Sym}}^{2}{\overline{Ann(V)}}$. 
In particular, when the divisor $\AD= (s=0)$ is reduced and irreducible, the map $\nu$ is injective if and only if the multiplication map $\mu$
is injective.
Indeed if $\mu$ is not injective, then there exists a non-zero element $w = s \otimes t + \sum_{j} u_j \otimes v_j$ in ${\rm{Sym}}^{2}{H^0(\Ab, \mathcal{L})}$ in the kernel of the multiplication map $\mu$, and by the above decomposition, we can assume $(\sum_{j} u_j \otimes v_j)_{|\sD}\neq 0$ in ${\rm{Sym}}^{2}{\overline{Ann(V)}}$. Since the diagram (\ref{diagmult1}) is commutative, this implies $\nu(\sum_{j} u_j \otimes v_j)_{|\sD}=0$, and thus $\nu$ is not injective.
On the other side, let us assume that $\sum_{j} u_j|_{\AD} \otimes v_j|_{\AD}$ is non-zero and belongs to the kernel of $\nu$, where $u_j$ and $v_j$ are non-zero holomorphic sections of $\mathcal{L}$. Then we have that $\mu(\sum_{j} u_j\otimes v_j) = \sum_j u_jv_j$ vanishes along $\AD$. Hence there exists $t \in H^0(\Ab, \mathcal{L})$ such that $st=\sum_j u_jv_j$. It follows that $\mu$ is not injective.

\subsection{On injectivity of the multiplication map}
Given now an abelian variety $(A, \AL)$, we want to give conditions which ensure the injectivity of the multiplication map $\mu$. We begin by fixing a decomposition of $\VSA$ for $\AL^2$ which, according to our previous discussion, induces a decomposition $K_1 \oplus K_2$ of $K:=Ker(\phi_{\AL^2})$. In particular, the same decomposition induces a decomposition $2K_1 \oplus 2K_2$ for the kernel of $\phi_{\AL}$.

Let us assume that $H$ is the non-degenerate hermitian form which corresponds to $\AL$ according to Appell-Humbert theorem. We recall that, by \cite[Lemma 1.2 p. 48]{BL}, $K(\AL) = \bigslant{\LATT(\AL)}{\Lambda}$ and $K(\AL^2) = \bigslant{\LATT(\AL^2)}{\Lambda}$, where
\begin{align*}
 \LATT(\AL) &= \{v \in \VSA \ \colon \ \IMh{H}(v,  \LATT) \subseteq \mZ \} \\
 \LATT(\AL^2) &= \{v \in \VSA \ \colon \ 2\IMh{H}(v,  \LATT) \subseteq \mZ \} 
\end{align*}
are lattices in $\VSA$, and $K(\AL)_i \cong \mZ_{d_1} \oplus \cdots \cdots  \mZ_{d_g}$ ($i=1,2$), where $(d_1 \cdots d_g)$ is the polarization type of $\AL$; 
see: \cite[Lemma 1.4 p. 50] {BL}. 

Moreover, $\LATT(\AL^2)$ contains the sublattice $\LATT(\AL)$, and the quotient is isomorphic to $\mZ_2^{g}$, with $2\LATT(\AL^2) = \LATT(\AL)$. On the other side, $K(\AL^2)_i \cong \mZ_{2d_1} \oplus \cdots \cdots  \mZ_{2d_g}$ it contains $K(\AL)_i$, and the quotient is isomorphic to $\mZ_2^{g}$.

\noindent 
Hence, the following is a basis for $H^0(\Ab, \AL)$:
\begin{equation}
\{ \ThfC{x}{\AL} \ \ : \ x \in 2K_1  \}
\end{equation}
\noindent
Let us denote by $Z_2:=\Ab[2] \cap K_1 \cong \mZ_2^g$. For every $(x_1, x_2) \in 2K_1 \oplus 2K_1$, and $(y_1, y_2) \in K_1 \oplus K_1$ such that $y_1 + y_2 = x_1$ and $y_1 - y_2 = x_2$, it holds the following multiplication formula:
\begin{equation*}
\mu(\ThfC{x_1}{\AL} \otimes \ThfC{x_2}{\AL}) = \sum_{z \in Z_2} \ThfC{y_2 + z}{\AL^2}(0)\ThfC{y_1 + z}{\AL^2}
\end{equation*}
Let us denote $Z_2':=Z_2 \cap 2K_1 \cong \mZ_2^{g-s}$, where $s$ is the number of odd indexes among $(d_1 \cdots d_g)$. Since $Z_2'$ is $\mZ_2$-subvector space there exists a unique $W$ complement of $Z_2'$ in $Z_2$. For a character $\rho\colon Z_2' \to \mathbb \mC^*$ of $ Z_2'$ we can define:

\begin{equation} \label{multmap}
\theta_{(x_1, x_2),\rho} := \sum_{z \in Z_2'}\rho(z)\ThfC{x_1 + z}{\AL} \otimes \ThfC{x_2 + z}{\AL}
\end{equation}
We have:
\begin{align*}
\mu(\theta_{(x_1, x_2),\rho}) &=  \sum_{z \in Z_2'}\rho(z)\mu(\ThfC{x_1 + z}{\AL} \otimes \ThfC{x_2 + z}{\AL}) \\
&=\sum_{z \in Z_2'}\sum_{t \in Z_2}\rho(z)\ThfC{y_2 + t}{\AL^2}(0) \cdot \ThfC{y_1 + t+ z}{\AL^2} \\
&=\sum_{z,z' \in Z_2'}\sum_{w \in W}\rho(z)\ThfC{y_2 + w + z'}{\AL^2}(0) \cdot \ThfC{y_1 + w+ z+ z'}{\AL^2} \\
&=\sum_{z,z' \in Z_2'}\sum_{w \in W}\rho(z)\rho(z')\ThfC{y_2 + w + z'}{\AL^2}(0)\cdot \ThfC{y_1 + w+ z}{\AL^2} \\
&=\sum_{w \in W} \left[\sum_{z' \in Z_2'}\rho(z')\ThfC{y_2 + w + z'}{\AL^2}(0)\right]\left[\sum_{z \in Z_2'}\rho(z) \ThfC{y_1 + w+ z}{\AL^2}\right] \\
&= \sum_{w \in W} C_{(y_2,w,\rho)}\cdot\theta^{\AL^2}_{(y_1, w, \rho)}
\end{align*}
where, with $t \in \bigslant{2K_1}{Z_2'}$ and $y \in \bigslant{K_1}{Z_2}$, 
\begin{align*}
C_{(t,w,\rho)} &:= \sum_{z \in Z_2'}\rho(z)\ThfC{t + w + z}{\AL^2}(0) \\
\theta^{\AL^2}_{(y, w, \rho)} &:= \sum_{z \in Z_2'}\rho(z) \ThfC{y + w+ z}{\AL^2}
\end{align*}
Now choose $U$ a transversal subset for $Z_2$ in $K_1$, that is $U$ is a subset of $K_1$ such that every (right or left) coset of $Z_2$ contains precisely one element of $U$. The set $U$ contains $\prod_{i=j}^g d_j$ elements, the quotient $2K_1/Z_2'$ contains $\prod_{i=j}^g d_j/2^{g-s}$ elements, while $\bigslant {2K_1 \times 2K_1}{\Delta_{Z_2'}}$ exactly $\prod_{i=j}^g d_j^2/2^{g-s}$, then the function $\psi \colon U \times \frac{2K_1}{Z_2'} \longrightarrow \bigslant {2K_1 \times 2K_1}{\Delta_{Z_2'}}$
%
which sends $(y,t)$ to $(y+t, y-t)$ is clearly a bijection, and we have
\begin{align*}
\mu(\theta_{(y+t, y-t),\rho}) &= \sum_{w \in W} C_{(t,w,\rho)}\cdot\theta^{\AL^2}_{(y, w, \rho)}
\end{align*}
Thus every character $\rho$ inside the group of characters $\widehat{Z_2'}$ gives a submatrix of the the matrix of multiplication map $\mu$ because we can write $\mu$ with respect to the two following basis of
$H^0(\Ab, \AL) \otimes H^0(\Ab, \AL)$ and $H^0(\Ab, \AL^2)$ respectively:
\begin{equation}\label{coppiaclasse}
\left\{ \theta_{(x_1, x_2),\rho} \ \ : \ [(x_1, x_2)] \in \bigslant{2K_1 \oplus 2K_1}{\Delta_{Z_2'}} \ \ \rho \in \widehat{Z_2'}  \right\}
\end{equation}
\noindent

\begin{equation}
\left\{ \theta^{\AL^2}_{(y, w, \rho)} \ \ : \ y \in U, \ \ w \in W, \ \ \rho \in \widehat{Z_2'}  \right\}
\end{equation}
We point out the reader that the couple $(x_1, x_2)$ in the definition (\ref{coppiaclasse}) of the basis of $H^0(\Ab, \AL) \otimes H^0(\Ab, \AL)$ is up to the diagonal action of $Z_2'$ since  $\theta_{(x_1+z, x_2+z),\rho} = \rho(z)\theta_{(x_1, x_2),\rho}$ for every $z$ in $Z_2'$.

By the matrix of $\mu$ with respect to the above bases it follows that the multiplication map is injective if and only if the restriction to the subspaces associated to the characters $\rho$ of $Z_2'$ is injective. We denote this restriction by
\begin{equation*}
\mu|_{\mathbb{V}_{y, \rho}} \colon \mathbb{V}_{y, \rho} \longrightarrow \mathbb{W}_{y, \rho}
\end{equation*}
where $\mathbb{V}_{y, \rho} := \left< \theta_{(y+t, y-t),\rho} \ \ : t \in  2K_1/Z_2' \right>$ and $\mathbb{W}_{y, \rho} := \oplus_{w \in W}\left< \theta^{\AL^2}_{(y, w, \rho)} \right>$. The change of sign of $t$ is obtained by exchanging $x_1$ and $x_2$ in definition (\ref{multmap}). Hence, the restriction of the multiplication map on the symmetric part $Sym(H^0(\Ab, \AL))$ can be described on the different blocks, each corresponding to $y$ and $\rho$, with the matrix:
\begin{align}\label{matrixMrho}
M_{\rho}&:= \left(C_{t,w,\rho}\right)_{t \in \pm 2K_1/Z_2', w \in W} 
\end{align}
This implies that $\mu$ is injective precisely when for every $y$ and for every character $\rho$ the matrix $M_{\rho}$ has maximal rank. 
\begin{thm}\label{polarizza}
 Let $(A, \AL)$ be a general $(1, 1, \cdots, 2, \cdots 2)$-polarized abelian variety of dimension $g$. Then the multiplication map $\mu$ is injective.
\end{thm}
\begin{proof}
Let us begin by the case in which the polarization type is $(2,\cdots 2)$.
For such a polarization, the matrix $M_{\rho}$ in \ref{matrixMrho} with $\rho \in Hom(\mZ_2^g, \mC^*)$ is just the scalar:
\begin{equation*}
C_{\rho} = \sum_{z \in A[2]}\rho(z)\ThfC{z}{\AL^2}(0)
\end{equation*}
Now inside the moduli space of $(2,2, \cdots 2)$-polarized abelian variety it is easy to conclude that $C_{\rho} \neq 0$ holds for the general abelian variety since it holds for
 the product of $2$-polarized elliptic curves. If we now take a general $(1,1,\cdots 1, 2, \cdots 2)$-polarization 
we can consider $(A', \mathcal{L}')$ a $(2,\cdots,2)$-polarized abelian variety with an isogeny $h \colon A' \longrightarrow A$ such that $h^*\AL = \mathcal{L}'$. Then the multiplication map $\mu$ on the sections of $\AL$ is just the restriction of the multiplication map
\begin{equation}
\mu_A \colon {\rm{Sym}}^{2}H^0(A, \mathcal{L}') \longrightarrow H^0(A', \mathcal{L}')
\end{equation}
to the symmetric product of the subvector space of the $Ker(h)$-invariant sections of $\mathcal{L}'$.
\end{proof}

\subsection{The Proof of Theorem [A]}
\label{Theorem A}
Let $\pi\colon \sX{\rightarrow}B$ be a family which satisfies extremal liftability conditions and such that for its general fiber the morphism $\phi_b\colon X_b\to Z_b\subset A$ is the desingularization of a divisor of a $(d_1,...,d_a)$ polarization. Since $A$ is general, we know that $\pi\colon \sX{\rightarrow}B$ is Nori trivial. By Theorem \ref{polarizza} the multiplication map given in the equation (\ref{iniettivobar}) is injective. By the Main Theorem, the claim follows.

\section{The case of maximal relative irregularity for a fibered surface}
Our basic reference for this last section is \cite[Section V]{PZ}.
Let $f\colon S\to B$ be a fibration over a smooth surface $S$ to a smooth curve $B$ with general fiber $F$. In order to understand the geometry of the fibration it is natural to try to obtain information by relating the invariants of $B$ and of the general fiber $F$ to those of $S$. For our purposes we only need to recall that if the relative irregularity $q(S)-g(B)=g(F)-1$ then there exists a hyperplane $V$ of $H^0(F,\omega_F)$ such that the standard restriction homomorphisms $H^0(S,\Omega^1_{S})\to H^0(F,\omega_{F})$ has $V$ as its own image. We need the following:

\begin{lem}\label{ultimo} Let $f\colon S\to B$ be a non-isotrivial fibration with general fiber of genus $\geq 3$ and such that  $q(S)-g(B)=g(F)-1$. Then the sublinear system induced by $V$ is base point free.
\end{lem}
\begin{proof} Take an infinitesimal deformation $\xi\in H^1(F,T_F)$ of $F$ given by the Kodaira-Spencer map. Assume that for the general fiber $F$ the image $V$ of $H^0(S,\Omega^1_{S})\to H^0(F,\omega_{F})$ has base points. Since $V$ is a hyperplane then Riemann-Roch theorem on curves implies that there exists a unique point $p_F\in F$ which is the base point of the linear system $|V|$. By the viceversa of the Adjoint theorem in the case of $1$-dimensional varieties, see \cite{CP}, it follows that $\xi$ is the Shiffer variation supported on $p_F$. This is a contradiction to \cite[Corollary 6.11]{AC}; see also \cite[Prop. 6.3.9]{victor1}
\end{proof}

Since $q(S)-g(B)=g(F)-1$ the Jacobians of the fibers have an Abelian variety $A'$ of dimension $g-1$ in common.
Let $B^0$ the open subscheme of $B$ where $f\colon S\to B$ is smooth. By shrinking to open subsets $U\subset B^0$, the family ${\rm{Alb}}(S)\times_{\rm{Alb}(B)}B\to B$ obtained by standard universal properties restricts to a family $p\colon \sA_U\to U$ whose fibers are all isomorphic to the dual $A$ of $A'$. Note that $p\colon \sA_U\to U$ is a family of polarised Abelian varieties where the fiber is always isomorphic to $A$ but the polarisation on $A\times\{b\}$ is given by:
$$
\Theta_b(\eta_1,\eta_2)=\int_{F_b} \phi^*_b(\eta_1)\wedge\phi^*_b(\eta_2)
$$
where $\phi_b\colon F_b\to A_b$ is given by the composition ${\rm{alb}}(S)\circ j_{F_b}\colon F_b\to {\rm{Alb}}(S)$, and $j_b\colon F_b\to S$ is the natural inclusion and $A_b$ is a translate of $A$ inside ${\rm{Alb}}(S)$.

\begin{thm}\label{vaiinfine}
Let $f\colon S\to B$ be a non-isotrivial fibration of maximal relative irregularity $> 2$ with fixed Abelian variety $A$ and let $B^0$ the open subscheme where it is smooth. 
If $\phi_b\colon F_b\to A$ has degree $1$ then the infinitesimal invariant associated to the basic cycle of the Albanese type family obtained by restricting over open subschemes $U\subset B^0$ is not zero.
\end{thm}
\begin{proof} The problem is local over $B$.
We can write $H^0(F_b,\Omega^1_{F_b})=V\oplus^{\perp} s\cdot\mathbb C$ where $V=\phi_b^{\star}H^0(A,\Omega^1_A)$ and $s$ is a nontrivial section. The decomposition is an orthogonal one with respect to the standard pairing on $F_b$.
  In particular ${\overline{{\rm{Ann}}(V)}}=s\cdot \mathbb C$. By contradiction assume that the infinitesimal invariant associated to an Albanese type family over a neighbourhood $U$ of $b$ is zero. Let $\xi\in H^1(F,T_F)$ be an infinitesimal deformation of $F$ given by the Kodaira-Spencer map of $f_{|f^{-1}(U)}f^{-1}(U)\to U$. By Theorem \ref{trasversalita} this means that all the adjoints obtained by $\xi$ and by $2$-dimensional subspaces $W\subset V$ belongs to $V^{\perp}=s\cdot\mathbb C$. This means that if $W=\langle \eta_1,\eta_2\rangle$ then there exists a constant $c\in\mathbb C$ such that $c \cdot s$ is an adjoint form associated to $W$, in other words $[c\cdot s]\in H^0(F,\omega_F)/W$ is the Massey product of  $\eta_{1}$ and $\eta_{2}$. This implies that if we take a general $\eta\in V$ and a general $2$-dimensional subspace $\langle \eta_1,\eta_2\rangle=W\subset V$ (in particular $\eta\not\in W$, here we need $q(S)-g(B)>2$, that is $g(F)\geq 4$), we can find a $\sigma\in W$ such that the Massey product of the $2$-dimensional subspace $\langle\eta,\sigma\rangle$ along $\xi$ is zero. Indeed if $[c_i\cdot s]\in H^0(F,\omega_F)/W_i$ is the Massey product of $W_i=\langle\eta,\eta_i\rangle$, $i=1,2$ then $\sigma=c_1\eta_2-c_2\eta_1\in W$. By the Adjoint theorem, it follows that $\xi=0$ if the linear system $\langle\eta,\sigma\rangle$ has no base points or that $\xi$ is supported on the base points of $\langle\eta,\sigma\rangle$. By the genericity of $\eta$ and $W$ it follows that $\xi$ is supported on the base points of the linear subsystem $V\subset H^0(F,\omega_F)$. By Lemma \ref{ultimo} we conclude that $\xi=0$. This means that $f\colon S\to B$ has constant moduli; a contradiction.
\end{proof}


\end{document}